\documentclass[12pt]{article}
\usepackage{amssymb,amsthm,amsmath,amsfonts,latexsym,tikz,hyperref}
\usepackage[hmargin=1in,vmargin=1in]{geometry}

\newtheorem{thm}{Theorem}[section]
\newtheorem{prop}[thm]{Proposition}
\newtheorem{cor}[thm]{Corollary}
\newtheorem{lem}[thm]{Lemma}
\newtheorem{conj}[thm]{Conjecture}
\newtheorem{exa}[thm]{Example}

\DeclareMathOperator{\nbc}{nbc}
\DeclareMathOperator{\isf}{isf}
\DeclareMathOperator{\ISF}{ISF}

\newcommand{\ben}{\begin{enumerate}}
\newcommand{\een}{\end{enumerate}}
\newcommand{\ble}{\begin{lem}}
\newcommand{\ele}{\end{lem}}
\newcommand{\bth}{\begin{thm}}
\renewcommand{\eth}{\end{thm}}
\newcommand{\bpr}{\begin{prop}}
\newcommand{\epr}{\end{prop}}
\newcommand{\bco}{\begin{cor}}
\newcommand{\eco}{\end{cor}}
\newcommand{\bcon}{\begin{conj}}
\newcommand{\econ}{\end{conj}}
\newcommand{\bde}{\begin{defn}}
\newcommand{\ede}{\end{defn}}
\newcommand{\bex}{\begin{exa}}
\newcommand{\eex}{\end{exa}}
\newcommand{\barr}{\begin{array}}
\newcommand{\earr}{\end{array}}
\newcommand{\btab}{\begin{tabular}}
\newcommand{\etab}{\end{tabular}}
\newcommand{\beq}{\begin{equation}}
\newcommand{\eeq}{\end{equation}}
\newcommand{\bea}{\begin{eqnarray*}}
\newcommand{\eea}{\end{eqnarray*}}
\newcommand{\bal}{\begin{align*}}
\newcommand{\bce}{\begin{center}}
\newcommand{\ece}{\end{center}}
\newcommand{\bpi}{\begin{picture}}
\newcommand{\epi}{\end{picture}}
\newcommand{\bpp}{\begin{picture}}
\newcommand{\epp}{\end{picture}}
\newcommand{\bfi}{\begin{figure} \begin{center}}
\newcommand{\efi}{\end{center} \end{figure}}
\newcommand{\bprf}{\begin{proof}}
\newcommand{\eprf}{\end{proof}\medskip}
\newcommand{\capt}{\caption}
\newcommand{\bsl}{\begin{slide}{}}
\newcommand{\esl}{\end{slide}}
\newcommand{\bfr}{\begin{frame}}
\newcommand{\efr}{\end{frame}}

\newcommand{\hqed}{\hfill \qed}

\newcommand{\eqed}[1]{$\textcolor{white}{\qed}\hfill{\dil#1}\hfill\qed$}

\newcommand{\hs}[1]{\hspace{#1}}
\newcommand{\hso}[1]{\hspace{-1pt}}
\newcommand{\vs}[1]{\vspace{#1}}

\newcommand{\emp}{\emptyset}

\newcommand{\sbs}{\subset}
\newcommand{\sbe}{\subseteq}

\newcommand{\setm}{\setminus}

\def\<{\langle}
\def\>{\rangle}

\newcommand{\ra}{\rightarrow}

\newcommand{\ka}{\kappa}
\newcommand{\la}{\lambda}

\newcommand{\1}{{\bf 1}}

\newcommand{\bx}{{\bf x}}

\newcommand{\bbN}{{\mathbb N}}
\newcommand{\bbP}{{\mathbb P}}

\newcommand{\bbR}{{\mathbb R}}

\newcommand{\bbZ}{{\mathbb Z}}

\newcommand{\cB}{{\cal B}}

\newcommand{\cC}{{\cal C}}

\newcommand{\cF}{{\cal F}}

\newcommand{\cH}{{\cal H}}
\newcommand{\cI}{{\cal I}}

\newcommand{\cN}{{\cal N}}

\newcommand{\cO}{{\cal O}}

\newcommand{\cS}{{\cal S}}

\newcommand{\dil}{\displaystyle}

\begin{document}
\pagestyle{plain}

\title{The Amazing Chromatic Polynomial
}
\author{
Bruce E. Sagan\\[-5pt]
\small Department of Mathematics, Michigan State University,\\[-5pt]
\small East Lansing, MI 48824-1027, USA, {\tt sagan@math.msu.edu}
}

\date{\today\\[10pt]
	\begin{flushleft}
	                                       \\[5pt]
	\end{flushleft}}

\maketitle

\begin{abstract}
Let $G$ be a combinatorial graph with vertices $V$ and edges $E$.  A proper coloring of $G$ is an assignment of colors to the vertices such that no edge connects two vertices of the same color.  These are the colorings considered in the famous Four Color Theorem.  It turns out that the number of proper colorings of $G$ using $t$ colors is a polynomial in $t$, called the chromatic polynomial of $G$.  This polynomial has many wonderful properties.  It also has the surprising habit of appearing in contexts which, a priori, have nothing to do with graph coloring.  We will survey three such instances involving acyclic orientations, hyperplane arrangements, and increasing forests.  In addition, connections to symmetric functions and algebraic geometry will be mentioned.
\end{abstract}


\section{Introduction}
\label{i}

We begin at the beginning with some basic definitions.  Let $\bbN$ and $\bbP$ be the nonnegative and positive integers, respectively.  Given $n\in\bbN$ we will use the notation $[n]=\{1,2,\ldots,n\}$ for the interval of the first $n$ positive integers.

A {\em (combinatorial) graph} is $G=(V,E)$ where $V$ is a finite set of vertices and $E$ is a set of edges, each edge connecting a pair of distinct vertices.  If edge $e$ goes between vertices $u$ and $v$ then we write $e=uv$ or $e=vu$ and call $u,v$ the {\em endpoints} of $e$.  
Alternatively, we say that $u$ and $v$ are {\em adjacent}.
As an example, on the left in Figure~\ref{G} we have a graph $G$ with vertex set $V=\{u,v,w,x\}$ and edge set $E=\{uv, ux, vx, vw\}$.

Given a set $S$, a {\em coloring} of $G$ by $S$ is a function $\ka:V\ra S$.  Figure~\ref{G} displays colorings of $G$ using the set $S=[3]$.  For example, the coloring in the middle has $\ka(v)=\ka(x)=1$ and
$\ka(u)=\ka(w)=2$.  (A coloring function need not be surjective.)  We say that a coloring $\ka$ is {\em proper} if each edge $e=uv$ of $G$ we have $\ka(u)\neq\ka(v)$.  Returning to Figure~\ref{G} we see that the middle coloring is not proper since $e=vx$ has  $\ka(v)=\ka(x)$.  But it is easy to check that the coloring $\ka'$ on the right is proper.

\bfi
\begin{tikzpicture}
\draw(1,-.8) node{$G$};
\draw (2,2)--(0,2)--(0,0)--(2,2)--(2,0);
\fill(0,0) circle(.1);
\draw(-.5,0) node{$x$};
\fill(2,0) circle(.1);
\draw(2.5,0) node{$w$};
\fill(0,2) circle(.1);
\draw(-.5,2) node{$u$};
\fill(2,2) circle(.1);
\draw(2.5,2) node{$v$};
\end{tikzpicture}
\hs{20pt}
\begin{tikzpicture}
\draw(1,-.8) node{$\ka$};
\draw (2,2)--(0,2)--(0,0)--(2,2)--(2,0);
\fill(0,0) circle(.1);
\draw(-.5,0) node{$1$};
\fill(2,0) circle(.1);
\draw(2.5,0) node{$2$};
\fill(0,2) circle(.1);
\draw(-.5,2) node{$2$};
\fill(2,2) circle(.1);
\draw(2.5,2) node{$1$};
\end{tikzpicture}
\hs{20pt}
\begin{tikzpicture}
\draw(1,-.8) node{$\ka'$};
\draw (2,2)--(0,2)--(0,0)--(2,2)--(2,0);
\fill(0,0) circle(.1);
\draw(-.5,0) node{$1$};
\fill(2,0) circle(.1);
\draw(2.5,0) node{$2$};
\fill(0,2) circle(.1);
\draw(-.5,2) node{$2$};
\fill(2,2) circle(.1);
\draw(2.5,2) node{$3$};
\end{tikzpicture}
\capt{A graph and two colorings \label{G}}
\efi

Proper colorings are the subject of the famous Four Color Theorem.  
If $S$ is a finite set then we use $\#S$ or $|S|$ for the cardinality of $S$.
The {\em chromatic number} of $G$, denoted $\chi(G)$, is the minimum $\#S$ such that there exists a proper coloring 
$\ka:G\ra S$.  Going back yet again to the graph in Figure~\ref{G}, we see that $\chi(G)=3$.  Indeed, $\ka'$ is a proper coloring with three colors.  And no proper coloring exists with fewer colors because of the triangle $\{uv, vx, xu\}$ which requires three colors.  Call a graph {\em planar} if it can be drawn in the Cartesian plane so that none of its edges cross.  Here is the landmark theorem of Appel and Haken (assisted by Koch).
\begin{thm}[The Four Color Theorem~\cite{ah:epm1,ah:epm2}]
If $G$ is a planar graph then 

\vs{8pt}

\eqed{
\chi(G)\le 4.
}
\end{thm}

This theorem is striking for several reasons.  First of all, there is no such bound for arbitrary graphs.  For consider the {\em complete graph}, $K_n$, which has $n$ vertices and all $\binom{n}{2}$ possible edges.  A drawing of $K_4$ is on the left in Figure~\ref{KC}.  (The crossing of two edges in the middle of the graph is not a vertex.)  Clearly $\chi(K_n)=n$ which can be as large as desired.  Second, this result had been conjectured for over $100$ years.  Finally, the proof was the first one to use computers in an integral way since the number of cases involved was too large for a human to check. The reader interested in the history of the Four Color Theorem is encouraged to consult Robin Wilson's excellent book~\cite{wil:fcs}.

Finding $\chi(G)$ is an extremal task since it involves minimization.  In this article we will be interested in a corresponding enumeration problem which was first studied by George Birkhoff~\cite{bir:dfn}.  Given a graph $G=(V,W)$ and $t\in\bbN$ the corresponding {\em chromatic polynomial} is
$$
P(G)=P(G;t)=\#\{\ka:V\ra[t] \mid 
\text{$\ka$ is proper}\}.
$$
Note that we could have used any set $S$ with $\#S=t$ in place of $[t]$ and gotten the same count. Also, it is not clear why we are calling this a polynomial.  But let's compute it for our perennial example in Figure~\ref{G}.
We will color the vertices in the order $u,v,w,x$.
Since $u$ is the first vertex colored, any of the $t$ elements in $[t]$ could be used.  So there are $t$ choices for $u$.  When we color $v$, it can be any color except the one used on $u$.  This gives $t-1$ choices.  Similarly there are $t-1$ choices for $w$.  Finally, when coloring $x$ we see that it is adjacent to the two previously colored vertices $u$ and $v$.  Furthermore, $u$ and $v$ are different colors since they are also adjacent.  This means that there are $t-2$ possible remaining colors for $x$.  Putting all these counts together, we see that the number of proper colorings of $G$  is
\begin{equation}
\label{P(G)ex}
 P(G;t)= t(t-1)(t-1)(t-2)= t^4 - 4 t^3 + 5 t^2 - 2 t.   
\end{equation}
Notice that this is a polynomial in $t$, the number of colors!  It turns out that this is always the case, which explains why $P(G;t)$ is called the chromatic polynomial.  

However, it is not true that one can always count the colorings as we did above and so obtain a polynomial whose roots are nonnegative integers.  To see what could go wrong, consider the {\em $n$-cycle}, $C_n$, which has $n$ vertices 
which can be ordered as $v_1,v_2,\ldots,v_n$ and $n$ edges $v_i v_{i+1}$ where $i$ is taken modulo $n$.  A copy of the cycle $C_4$ is shown on the right in Figure~\ref{KC}.  Let us now try coloring $C_4$ in the order $u,v,w,x$.  As before, there are $t$ choices for $u$ and $t-1$ for $v$ and $w$.  But we now have a problem trying to color $x$.  For this vertex has edges to both $u$ and $w$.  But since $u$ and $w$ are not adjacent, we do not know whether they were assigned the same color or not.  There is an elegant way around this difficulty called deletion-contraction which we will discuss in the next section.  

We should also note that there is a simple relationship between the chromatic number and the chromatic polynomial.  Specifically, $\chi(G)$ is the smallest positive integer such that $P(G;\chi(G))\neq 0$.  To see this note that if $0<t<\chi(G)$ then, by definition of $\chi$, there are no proper colorings of $G$ with $t$ colors.  So, since $P(G;t)$ counts the number of such colorings, it must evaluate to zero.  On the other hand, there must be at least one proper coloring of $G$ with $t=\chi(G)$ colors.  It follows that
$P(G;\chi(G))> 0$.

The rest of this paper is organized as follows.  In the next section we will introduce the method of deletion-contraction.  It will be used to prove that $P(G;t)$ is always a polynomial in $t$ as well as to compute the chromatic polynomial of $C_4$.  We will also exhibit a combinatorial interpretation of the coefficients of $P(G;t)$ in terms of certain sets of edges of $G$ which are said to contain no broken circuit.  One of the amazing things about $P(G;t)$ is that it often appears in contexts which seem to have nothing to do with graph coloring, or even with graphs!  Three examples of this will be given in Section~\ref{ta}.  We will end with a section giving more information about the chromatic polynomial, including connections with symmetric functions and with algebraic geometry.

\bfi
\begin{tikzpicture}
\draw(1,-.8) node{$K_4$};
\draw (2,2)--(0,2)--(0,0)--(2,2)--(2,0) (0,0)--(2,0)--(0,2);
\fill(0,0) circle(.1);
\draw(-.5,0) node{$x$};
\fill(2,0) circle(.1);
\draw(2.5,0) node{$w$};
\fill(0,2) circle(.1);
\draw(-.5,2) node{$u$};
\fill(2,2) circle(.1);
\draw(2.5,2) node{$v$};
\end{tikzpicture}
\hs{40pt}
\begin{tikzpicture}
\draw(1,-.8) node{$C_4$};
\draw (2,2)--(0,2)--(0,0)--(2,0)--(2,2);
\fill(0,0) circle(.1);
\draw(-.5,0) node{$x$};
\fill(2,0) circle(.1);
\draw(2.5,0) node{$w$};
\fill(0,2) circle(.1);
\draw(-.5,2) node{$u$};
\fill(2,2) circle(.1);
\draw(2.5,2) node{$v$};
\end{tikzpicture}
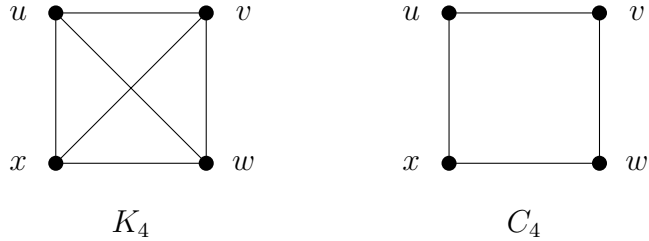
\capt{The complete graph $K_4$ and the cycle $C_4$ \label{KC}}
\efi


\section{Why is it a polynomial?}
\label{wp}

In this section we will prove that $P(G;t)$ is actually a polynomial in $t$ by using the deletion-contraction method.  We will also define NBC (no broken circuit) sets and use them to describe the coefficients of this polynomial.

\bfi
\begin{tikzpicture}
\draw(1,-.8) node{$G$};
\draw (2,2)--(0,2)--(0,0)--(2,2)--(2,0);
\fill(0,0) circle(.1);
\draw(-.5,0) node{$x$};
\fill(2,0) circle(.1);
\draw(2.5,0) node{$w$};
\fill(0,2) circle(.1);
\draw(-.5,2) node{$u$};
\fill(2,2) circle(.1);
\draw(2.5,2) node{$v$};
\end{tikzpicture}
\hs{20pt}
\begin{tikzpicture}
\draw(1,-.8) node{$G\setm vx$};
\draw (2,0)--(2,2)--(0,2)--(0,0);
\fill(0,0) circle(.1);
\fill(0,0) circle(.1);
\draw(-.5,0) node{$x$};
\fill(2,0) circle(.1);
\draw(2.5,0) node{$w$};
\fill(0,2) circle(.1);
\draw(-.5,2) node{$u$};
\fill(2,2) circle(.1);
\draw(2.5,2) node{$v$};
\end{tikzpicture}
\hs{20pt}
\begin{tikzpicture}
\draw(1,-.8) node{$G/vx$};
\draw (2,0)--(0,2);
\fill(2,0) circle(.1);
\draw(2.5,0) node{$w$};
\fill(0,2) circle(.1);
\draw(-.5,2) node{$u$};
\fill(1,1) circle(.1);
\draw(1.3,1.3) node{$v'$};
\end{tikzpicture}
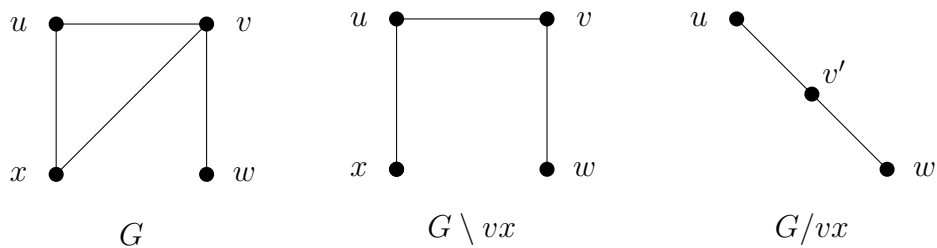
\capt{Deletion and contraction \label{DCfig}}
\efi

If $G=(V,E)$ is a graph and $e\in E$ then {\em deleting $e$} from $G$ gives a graph $G\setm e$ on the same vertex set and with edges the set difference $E\setm\{e\}$. The central graph in Figure~\ref{DCfig} shows the result of deleting the edge $e=vx$ from our canonical graph $G$.  The {\em contraction of $e=vx$} in $G$, denoted $G/e$, is obtained by collapsing the edge to a new vertex $v'$ where $v'$ is adjacent to all the vertices which were adjacent either the $v$ or $x$.  All other vertices and edges stay the same in $G/e$.  The graph on the right in Figure~\ref{DCfig} illustrates $G/vx$.  Note also that the two edges $uv$ and $ux$ in $G$ become a single edge $uv'$ in $G/vx$.  Since both $G\setm e$ and $G/e$ have fewer edges than $G$, the next result is a perfect recursion for induction on the number of edges.
\begin{lem}[Deletion-Contraction Lemma]
Given a graph $G=(V,E)$ and $e\in E$ we have
$$
P(G) = P(G\setm e) - P(G/e).
$$
\end{lem}
\begin{proof}
We will prove this result in the form
\begin{equation}
\label{P(G-e)}
   P(G\setm e)= P(G) + P(G/e).  
\end{equation}
Suppose $e=vx$.  Since $v$ and $x$ are no longer adjacent in $G\setm e$, the proper colorings $\ka$ of this graph are of two types: those where 
$\ka(v)\neq\ka(x)$ and those where $\ka(v)=\ka(x)$.
But if $\ka$ is proper on $G\setm e$ and also satisfies
$\ka(v)\neq\ka(x)$ then $\ka$ is a proper coloring of $G$.  Conversely, every proper coloring of $G$ gives rise to a proper coloring of $G\setm e$ where $\ka(v)\neq\ka(x)$.  So these colorings of $G\setm e$ are counted by $P(G)$.

Now consider the proper colorings $\ka$ of $G\setm e$ with $\ka(v)=\ka(x)$.  Such a coloring can be lifted to a proper coloring $\ka'$ of $G/e$ where $\ka'(w)=\ka(w)$ for $w\neq v'$, and $\ka'(v')$ is the common color assigned to $v$ and $x$ by $\ka$.  As with colorings of the first type, this produces a bijections between the proper colorings of $G/e$ and those of  $G\setm e$ with $\ka(v)=\ka(x)$.  It follows that the number of colorings in this case is $P(G/e)$.  Combining the two possibilties yields equation~\eqref{P(G-e)} and proves the lemma.
\end{proof}

It is now easy to prove Birkhoff's fundamental result about the chromatic polynomial.
\begin{thm}[\cite{bir:dfn}]
Let $G=(V,E)$ be a graph with $\#V=n$.  Then $P(G;t)$ is a polynomial in $t$ of degree
$$
\deg P(G;t) = n.
$$
\end{thm}
\begin{proof}
We will induct on $m=\#E$.  If $m=0$ then $G$ is just a set of $n$ vertices.  Since there are no edges, each vertex can be colored independently in any of $t$ ways.  So in this case $P(G;t)=t^n$ which certainly satisfies the requirements of the theorem.

Now suppose that $m>0$.  So $E$ is nonempty and pick any $e\in E$.  By the Deletion-Contraction Lemma, $P(G)=P(G\setm e)-P(G/e)$.  By induction, $P(G\setm e)$ is a polynomial in $t$ of degree $n$ since $G\setm e$ and $G$ have the same number of vertices.  We also have that $P(G/e)$ is a polynomial in $t$.  But it has one fewer vertex than $G$ and so is of degree $n-1$.  The proof is completed by observing that the difference of a polynomial of degree $n$ and one of degree $n-1$ is a polynomial of degree $n$.
\end{proof}

Deletion-contraction is also a useful tool when it comes to computing chromatic polynomials.  Recall that we were not able to compute $P(C_4)$ for the $4$-cycle in the previous  section.  But after deleting and contracting  one of its edges $e$, the computation is reduced to graphs where we can apply the vertex-by-vertex technique used earlier.  Specifically,

\begin{align*}
\begin{tikzpicture}
\draw(-.5,.5) node{$P\left(\rule{0pt}{20pt}\right.$};
\draw(.5,1.3) node{$e$};
\fill(0,0) circle(.1);
\fill(1,0) circle(.1);
\fill(0,1) circle(.1);
\fill(1,1) circle(.1);
\draw (0,0)--(1,0)--(1,1)--(0,1)--(0,0);
\draw(1.5,.5) node{$\left.\rule{0pt}{20pt}\right)$};
\end{tikzpicture}
&\raisebox{20pt}{=}
\begin{tikzpicture}
\draw(-.5,.5) node{$P\left(\rule{0pt}{20pt}\right.$};
\fill(0,0) circle(.1);
\fill(1,0) circle(.1);
\fill(0,1) circle(.1);
\fill(1,1) circle(.1);
\draw (0,1)--(0,0)--(1,0)--(1,1);
\draw(1.5,.5) node{$\left.\rule{0pt}{20pt}\right)$};
\draw(2,.5) node{$-$};
\end{tikzpicture}
\begin{tikzpicture}
\draw(-.5,.5) node{$P\left(\rule{0pt}{20pt}\right.$};
\fill(0,0) circle(.1);
\fill(1,0) circle(.1);
\fill(.5,1) circle(.1);
\draw (0,0)--(1,0)--(.5,1)--(0,0);
\draw(1.5,.5) node{$\left.\rule{0pt}{20pt}\right)$};
\end{tikzpicture}
\\
&= t(t-1)^3 - t(t-1)(t-2)\\[5pt]
&=t(t-1)(t^2-3t+3).
\end{align*}
Note that, unlike the polynomial computed in Section~\ref{i}, this one has complex roots.

\bfi
\begin{tikzpicture}
\draw(1,-.8) node{$G$};
\draw (2,2)--(0,2)--(0,0)--(2,2)--(2,0);
\fill(0,0) circle(.1);
\fill(2,0) circle(.1);
\fill(0,2) circle(.1);
\fill(2,2) circle(.1);
\draw(1,2.3) node{$b$};
\draw(-.3,1) node{$c$};
\draw(1.2,.8) node{$d$};
\draw(2.3,1) node{$e$};
\end{tikzpicture}
\hs{20pt}
\begin{tikzpicture}
\draw(1,-.8) node{$C$};
\draw (2,2)--(0,2)--(0,0)--(2,2);
\fill(0,0) circle(.1);
\fill(0,2) circle(.1);
\fill(2,2) circle(.1);
\draw(1,2.3) node{$b$};
\draw(-.3,1) node{$c$};
\draw(1.2,.8) node{$d$};
\end{tikzpicture}
\hs{20pt}
\begin{tikzpicture}
\draw(1,-.8) node{$B$};
\draw (0,2)--(0,0)--(2,2);
\fill(0,0) circle(.1);
\fill(0,2) circle(.1);
\fill(2,2) circle(.1);
\draw(-.3,1) node{$c$};
\draw(1.2,.8) node{$d$};
\end{tikzpicture}
\capt{A broken circuit \label{BC}}
\efi

Since $P(G;t)$ is a polynomial, one would like a description of its coefficients.  This was done by Hassler Whitney~\cite{whi:lem}.    Fix a total ordering $e_1<e_2<\ldots<e_m$ of the edge set $E$.  A {\em broken circuit} of $G$ is a subset  $B\sbs E$ obtained by removing the smallest edge from the edge set of some cycle of $G$.  Consider our usual graph $G$ with edges labeled as in Figure~\ref{BC} and ordered by $b<c<d<e$.
Then $G$ has a unique cycle $C$ with edges $\{b,c,d\}$.
The corresponding broken circuit is $\{c,d\}$.  Call  $A\sbe E$ an {\em NBC set} (short for {\em no broken circuit set}) if $A$ does not contain any broken circuit of $G$.  In our example, these are exactly the edge sets not containing $\{c,d\}$.  Let
$$
\nbc_k(G) =\#\{A\sbe E \mid \text{$A$ is an NBC set with $k$ edges}\}.
$$
Making a chart of these numbers for our example graph gives
$$
\begin{array}{c|c|c}
k   & \text{NBC sets with $k$ edges} & \nbc_k(G)\\
\hline \hline
 0    &  \emp                                           & 1 \rule{0pt}{15pt}\\[5pt]
 1    & \{b\},\ \{c\},\ \{d\},\ \{e\}                   & 4\\[5pt]
 2    & \{b,c\},\ \{b,d\},\ \{b,e\},\ \{c,e\},\ \{d,e\} & 5 \\[5pt]
 3    & \{b,c,e\},\ \{b,d,e\}                           & 2 \\[5pt]
 4    & \text{none}                                     & 0
\end{array}
$$
Comparing the last column to the coefficients of $P(G;t)$ as calculated in~\eqref{P(G)ex}, the reader should have a conjecture in mind.
\begin{thm}[\cite{whi:lem}]
\label{whi}
For any graph $G=(V,E)$ with $\#V=n$ and any total order on $E$ we have 

\vs{8pt}

\eqed{P(G;t)=\sum_{k=0}^n (-1)^k \nbc_k(G)\ t^{n-k}.}.
\end{thm}

This theorem is remarkable since it implies that the numbers $\nbc_k(G)$ do not depend on the ordering given to the edge set, even though the actual NBC sets may be different.  It also makes calculating certain coefficients of $P(G;t)$ very easy.  For example $\nbc_0(G)=1$ because of the empty edge set.  So $P(G;t)$ is monic.  Furthermore, since a cycle has at least three edges, any broken circuit has at least two.  It follows that all single edges are NBC and thus the coefficient of $t^{n-1}$ is $-|E|$.


\section{Three applications}
\label{ta}

We will now look at three theorems where the chromatic polynomial makes a surprising appearance.  These are results of Stanley~\cite{sta:aog} on acyclic orientations, Zaslavsky~\cite{zas:fua} on hyperplane arrangements, and Hallam and Sagan~\cite{hs:fcp} on increasing forests (later improved upon by Hallam, Martin and Sagan~\cite{hms:isf}).

\bfi
\begin{tikzpicture}
\draw(1,-.8) node{$O$};
\draw[->] (0,2)-- (1,2);
\draw[->] (2,2)-- (1,1);
\draw[->] (0,0)-- (0,1);
\draw[->] (2,0)-- (2,1);
\draw (1,2)--(2,2) (1,1)--(0,0) (0,1)--(0,2) (2,1)--(2,2);
\fill(0,0) circle(.1);
\draw(-.5,0) node{$x$};
\fill(2,0) circle(.1);
\draw(2.5,0) node{$w$};
\fill(0,2) circle(.1);
\draw(-.5,2) node{$u$};
\fill(2,2) circle(.1);
\draw(2.5,2) node{$v$};
\end{tikzpicture}
\hs{40pt}
\begin{tikzpicture}
\draw(1,-.8) node{$O'$};
\draw[->] (0,2)-- (1,2);
\draw[->] (0,0)-- (1,1);
\draw[->] (0,0)-- (0,1);
\draw[->] (2,2)-- (2,1);
\draw (1,2)--(2,2) (1,1)--(2,2) (0,1)--(0,2) (2,1)--(2,0);
\fill(0,0) circle(.1);
\draw(-.5,0) node{$x$};
\fill(2,0) circle(.1);
\draw(2.5,0) node{$w$};
\fill(0,2) circle(.1);
\draw(-.5,2) node{$u$};
\fill(2,2) circle(.1);
\draw(2.5,2) node{$v$};
\end{tikzpicture}
\capt{Two orientations \label{O}}
\efi

A {\em digraph} or {\em directed graph} is $D=(V,A)$ consisting of a set of vertices $V$ and a set of arcs $A$ where each arc goes from one vertex to another.  If arc $a$ goes from vertex $u$ to vertex $v$ then we write $a=\vec{uv}$.  For example, the arc set for the digraph $O$ in Figure~\ref{O} is 
$A=\{\vec{uv},\ \vec{vx},\ \vec{xu},\ \vec{wv}\}$.
A {\em directed cycle} $v_1,v_2,\ldots,v_n$ in a digraph is defined analogously to a cycle in a graph where one insists that there is an arc from $v_i$ to $v_{i+1}$ for all $i$ modulo $n$.  A digraph without cycles is said to be {\em acyclic}.  In Figure~\ref{O}, the digraph $O$ has a cycle $u,v,x$ while $O'$ is acyclic.

Given a graph $G=(V,E)$, an {\em orientation} of $G$ is a digraph obtained by replacing each edge $uv$ by one of its two possible orientations $\vec{uv}$ or $\vec{vu}$.  The two digraphs in Figure~\ref{O} are both orientations of our standard example graph $G$ in Figure~\ref{G}.  Clearly the number of orientations of $G$ is $2^{\#E}$.  But what if we require the orientations to be acyclic?  Let 
$$
\cO(G) = \{ O \mid \text{$O$ is an acyclic orientation of $G$}\}.
$$
Returning to our standard example, there are $2^3$ orientations of the (undirected) cycle $u,v,x$.  Of these, two of them create a directed cycle, one going clockwise and the other counterclockwise.  So there are $2^3-2=6$ acyclic orientations of this part of $G$.  As for the edge $vw$, it can be oriented either way without producing  a cycle.  So for this graph $\#\cO(G)=6\cdot 2 = 12$.  We will now do something completely crazy.  Let's plug  $t=-1$ into the chromatic polynomial of $G$ as  computed in~\eqref{P(G)ex}.  This gives
$$
P(G;-1)= (-1)^4 - 4(-1)^3 + 5(-1)^2 - 2(-1) = 12.
$$
This is the same $12$ as the previous one.
\begin{thm}[\cite{sta:aog}]
\label{sta}
For any graph $G=(V,E)$ with $\#V=n$ we have

\vs{8pt}

\eqed{
P(G;-1) = (-1)^n \#\cO(G).
}
\end{thm}

It is not at all clear what it means to color a graph with $-1$ colors.  However, we can make some combinatorial sense of this result.  By Theorem~\ref{whi} we have
$$
P(G;-1)=\sum_{k=0}^n (-1)^k \nbc_k(G)\ (-1)^{n-k} 
=(-1)^n \sum_{k=0}^n \nbc_k(G).
$$
So one could give a combinatorial proof of Theorem~\ref{sta} by constructing a bijection between the NBC sets of $G$ and its acyclic orientations as was done by Blass and Sagan~\cite{bs:bpt}.

\bfi
\begin{tikzpicture}
\draw(-1,-2)--(1,2) (-2,2)--(2,-2);
\draw(1.3,1) node{$y=2x$};
\draw(-2,1) node{$y=-x$};
\end{tikzpicture}
\capt{A hyperplane arrangement in $\bbR^2$ \label{cH}}
\efi

We now turn to hyperplane arrangements.  Let $\bbR$ be the real numbers.  A {\em hyperplane} $H$ in $\bbR^n$ is a subspace of dimension $n-1$.  Note that, as a subspace, a hyperplane must go through the origin.  A {\em hyperplane arrangement} is just a finite set of hyperplanes $\cH=\{H_1,H_2,\ldots,H_k\}$.  For example, in $\bbR^2$ the hyperplanes are just lines through $(0,0)$ and the arrangement $\cH=\{y=2x,\ y=-x$\} is shown in Figure~\ref{cH} (without the coordinate axes for clarity in what comes later).  The {\em regions} of an arrangement $\cH$ are the connected components that remain after one removes the hyperplanes of the arrangement from $\bbR^n$.  Let $R(\cH)$ be the set of regions of $\cH$.  So in Figure~\ref{cH}, $R(\cH)$ consists of four regions.  Indeed, any arrangement of $k$ hyperplanes in $\bbR^2$ has $2k$ regions, but things get more complicated in $\bbR^n$.

At first blush, these concepts seem to have nothing to do with the chromatic polynomial.  But wait!  Suppose $G=(V,E)$ is a graph with $V=[n]$.  Note that we are now using an interval of integers for the labels of the vertices, not for the colors of a coloring.  Write 
$$
\bbR^n=\{(x_1,x_2,\ldots,x_n) \mid 
\text{$x_i\in\bbR$ for all $i\in[n]$}\}.
$$
We can now associate with $G$ an arrangement of hyperplanes in $\bbR^n$ defined by
$$
\cH(G) = \{ x_i=x_j \mid ij\in E\}.
$$
So each edge of $G$ gives rise to a hyperplane gotten by setting the coordinate functions of its endpoints equal.  As an example,
suppose $G$ has $V=[3]$ and $E=\{12,23\}$.  Then the corresponding arrangement would be 
$\cH(G)=\{x_1=x_2,\ x_2=x_3\}$ in $\bbR^3$.  Notice that the number of regions of $\cH(G)$ is $4$ in this case.  It is also easy to see that $P(G;t)=t(t-1)^2$ so that $P(G;-1)=-4$.  Again, this is not an accident.
\begin{thm}[\cite{zas:fua}]
\label{zas}
For any graph $G=(V,E)$ with $V=[n]$ we have

\vs{8pt}

\eqed{
P(G;-1) = (-1)^n \cdot \#R(\cH(G)).
}
\end{thm}

Seeing the two previous theorems back-to-back, the reader may suspect that they are related.  In fact there is a bijection between acyclic orientations  of $G$ with vertices labeled by $[n]$ and regions of its hyperplane arrangement.  Every hyperplane $x_i=x_j$ determines two half-spaces, namely $x_i<x_j$ and $x_i>x_j$.  Consider these as corresponding to the arcs $\vec{ij}$ and $\vec{ji}$, respectively.  One can then show that an orientation $O$ of $G$ is acyclic if and only if the intersection of the associated half-spaces is a region of $\cH(G)$.

For our third application, we will need a few more definitions from graph theory.  A {\em subgraph} of $G=(V,E)$ is a graph $G'=(V',E')$ with  $V'\sbe V$ and $E'\sbe E$.  We say $G'$  is {\em spanning} if $V'=V$.  
In this case, we often identify $G$ with its edge set since the set of vertices is fixed.
In Figure~\ref{F} we see our usual example graph with the vertices relabeled by $[4]$, as well as two spanning subgraphs $F=\{12,24\}$ and $F'=\{14,24\}$.  A {\em  path from $u$ to $v$} in $G$ is a sequence of distinct vertices $P:u=v_1,v_2,\ldots,v_n=v$ where $v_i v_{i+1} \in E$ for $i\in[n-1]$. Returning to Figure~\ref{F} we see that $P:4,1,2,3$ is a path from $4$ to $3$ in $G$.  Call $G$ {\em connected} if for every pair of vertices $u,v$ there is a path from $u$ to $v$.  More generally, the {\em components} of $G$ are the connected subgraphs  which are maximal with respect to inclusion.  
So a connected graph has one component.
The graph $G$ in Figure~\ref{F} is connected, but the subgraphs $F$ and $F'$ are not.  Both of these subgraphs have two components.

\bfi
\begin{tikzpicture}
\draw(1,-.8) node{$G$};
\draw (2,2)--(0,2)--(0,0)--(2,2)--(2,0);
\fill(0,0) circle(.1);
\draw(-.5,0) node{$4$};
\fill(2,0) circle(.1);
\draw(2.5,0) node{$3$};
\fill(0,2) circle(.1);
\draw(-.5,2) node{$1$};
\fill(2,2) circle(.1);
\draw(2.5,2) node{$2$};
\end{tikzpicture}
\hs{20pt}
\begin{tikzpicture}
\draw(1,-.8) node{$F$};
\draw (0,0)--(2,2)--(0,2);
\fill(0,0) circle(.1);
\draw(-.5,0) node{$4$};
\fill(2,0) circle(.1);
\draw(2.5,0) node{$3$};
\fill(0,2) circle(.1);
\draw(-.5,2) node{$1$};
\fill(2,2) circle(.1);
\draw(2.5,2) node{$2$};
\end{tikzpicture}
\hs{20pt}
\begin{tikzpicture}
\draw(1,-.8) node{$F'$};
\draw (0,2)--(0,0)--(2,2);
\fill(0,0) circle(.1);
\draw(-.5,0) node{$4$};
\fill(2,0) circle(.1);
\draw(2.5,0) node{$3$};
\fill(0,2) circle(.1);
\draw(-.5,2) node{$1$};
\fill(2,2) circle(.1);
\draw(2.5,2) node{$2$};
\end{tikzpicture}
\capt{A graph $G$ and two spanning forests $F,F'$ \label{F}}
\efi

A graph $F$ is a {\em forest} if it contains no (undirected) cycles.  The components of $F$ are called {\em trees}.  Trees can be characterized by the fact that for any pair of vertices there is a unique path between them,  So $F$ and $F'$ in Figure~\ref{F} are spanning forests of $G$.  Let $F$ be a forest with  $V=[n]$ so that we can compare the size of the vertex labels.
Say that $F$ is {\em increasing} if the vertex labels of any path starting at the minimum vertex in its component tree form an increasing sequence.  
So the forest $F$ in Figure~\ref{F} is increasing.  Indeed, in the tree with one vertex, there is only the path $3$ which is trivially increasing.   Note that a similar argument shows that any tree with only one or two vertices satisfies the increasing condition.  As far as the tree with three vertices in $F$, all paths from the minimum vertex $1$ are subpaths of $1,2,4$.  And this is an increasing sequence.  On the other hand, the forest $F'$ is not increasing since $1,4,2$ is a path starting at $1$ which is not an increasing sequence.  

Given $G$ with vertices $[n]$, consider the integers
$$
\isf_k(G) = \#\{ F \mid \text{$F$ is an increasing spanning forest of $G$ with $k$ edges}\}
$$
with generating function
$$
\ISF(G)=\ISF(G;t) = \sum_{k=0}^n (-1)^k\isf_k(G)\ t^{n-k},
$$
Note that, although our notation doesn't show it, $\isf_k(G)$ depends on how the vertices of $G$ are labeled.  Also, it is not clear why we have introduced the signs in $\ISF(G)$ or why we made $\isf_k(G)$ the coefficient of $t^{n-k}$ rather than $t^k$.  But this will become obvious shortly.  Let us compute the generating function for $G$ as in Figure~\ref{F}.  First of all $\isf_0(G)=1$ because the spanning graph with no edges has only single vertex trees which are all increasing.  Next $\isf_1(G)=\#E=4$ since any single edge is increasing.  There are $\binom{4}{2}=6$ ways to choose a spanning forest with two edges, of which only the $F'$ in Figure~\ref{F} is not increasing.  So 
$\isf_2(G)=6-1=5$.  Similarly, one computes that $\isf_3(G)=2$.  Finally, $\isf_4(G)=0$ since the only spanning subgraph of $G$ with $4$ edges is $G$ itself which is not even a forest.  Putting everything together we obtain
$$
\ISF(G;t) = t^4 - 4 t^3 + 5 t^2 - 2 t,
$$
a polynomial which we have seen previously!

But before we explore the connection between $\ISF(G;t)$ and $P(G;t)$, we wish to mention a nice factorization of the former.  As usual, suppose $G=(V,E)$ has $V=[n]$ and define the following edges sets
$$
E_j=E_j(G) = \{ij\in E \mid i<j\}
$$
for $j\in[n]$.  Note that we always have $E_1=\emp$ since there are no vertices with label smaller than $1$.  Also the $E_j$ partition $E$ in that $E=\uplus_{j\in[n]} E_j$.  In our usual example
$$
E_1 = \emp,\ E_2 =\{12\},\ E_3=\{23\},\ E_4=\{14,24\}.
$$
These sets give rise to the  polynomial
$$
\prod_{j=1}^4 (t-\#E_j) = (t-0)(t-1)(t-1)(t-2) = t^4 - 4 t^3 + 5 t^2 - 2 t
$$
which by now should be very familiar!
This is explained by the next result.
\begin{thm}[\cite{hs:fcp}]
\label{Ej}
For any graph  $G=(V,E)$ with $V=[n]$ we have the following.
\begin{enumerate}
    \item[(a)]  Subgraph $F$ of $G$ is an increasing spanning forest if and only if $|F\cap E_j|\le 1$ for all $j\in[n]$.
    \item[(b)]  We have
    
    \vs{4pt}
    
    \eqed{\ISF(G;t) = \prod_{j=1}^n (t-\#E_j).}
\end{enumerate}
\end{thm}

Note that part (b) of this theorem follows directly from part (a).  For expanding the product shows that the coefficient of $t^{n-k}$ is (up to sign) the number of way of choosing $k$ edges of $G$ with no two coming from the same $E_j$.

There are at least two reasons why one can not always have $\ISF(G;t)=P(G;t)$.  For one thing, the former depends on which labels are given to the vertices while the latter does not.  And we have seen that $P(C_4)$ has complex roots, while the previous result shows that the roots of $\ISF(G)$ are always nonnegative integers.  So the question becomes when are the two polynomials equal?  The answer has to do with the notion of a perfect elimination ordering.  

Given $G=(V,E)$ then the graph {\em induced} by $W\sbe V$ is
$$
G[W] = G\setm W'
$$
where $W'$ is the complement of $W$ in $V$ and deletion of multiple edges is defined just as it was for a single edge.  Another description of $G[W]$ is that it is the subgraph of $G$ with vertex set $W$ and all possible edges of $G$ whose endpoints are in $W$.  For example, in Figure~\ref{F} the graph $G[1,2,4]$ is a $3$-cycle while $G[2,3,4]$ is the path $3,2,4$.

We say that $G=(V,E)$ has a {\em perfect elimination ordering} if there is an ordering of $V$ as $v_1,v_2,\ldots,v_n$ such that for all $j\in[n]$ the induced subgraph $G[V_j]$ is complete where
$$
V_j =\{v_i \mid \text{$i<j$ and $v_i v_j\in E$}\}.
$$
This definition may seem strange to those seeing it for the first time.  But it is a useful concept, for example, as a characterization of chordal graphs.  Suppose the vertices of the graph $G$ in Figure~\ref{F} are ordered in the natural way as $1,2,3,4$.  Then
the corresponding $V_j$ are just the vertices smaller than $j$ in the edges of $E_j$ which gives
$$
V_1 = \emp,\ V_2 =\{1\},\ V_3=\{2\},\ V_4=\{1,2\}.
$$
Clearly the graphs $G[V_j]$ for $j\le 3$ are complete since $G[\emp]$ is the empty graph, and $G[v]$ is just  $v$ for any single vertex $v$.  Finally, $G[V_4]$ is the edge $12$ which is also a complete graph.  So we have a perfect elimination ordering which presages the next theorem.
\begin{thm}[\cite{hs:fcp}]
\label{peo}
Let $G$ be a graph with $V=[n]$.  We have $P(G;t)=\ISF(G;t)$ if and only if the natural order on $[n]$ is a perfect elimination ordering of $G$.\hqed
\end{thm}


\section{Going further}
\label{gf}

We will now discuss even more striking results related to the chromatic polynomial.  These will include a generalization to symmetric functions and connections with algebraic geometry.

\subsection{Credit where credit is due}

For pedagogical reasons, the results presented in the previous section only gave a partial picture of the authors' contributions.  Here we will take a wider view.

Since the chromatic polynomial of $G=(V,E)$ at $t=-1$ has a nice combinatorial interpretation, one might ask what happens at negative integers in general.
Let $t\in\bbP$ and $\ka:V\ra[t]$ a coloring which may not be proper.
Also consider an acyclic orientation $O=(V,A)$ of $G$.  We say that $O$ and $\ka$
are {\em compatible} if
$$
\vec{uv}\in A \implies \ka(u)\le\ka(v).
$$
So $O$ is like a gradient vector field, always pointing from lower to higher values of $\ka$.  Stanley's full theorem is as follows.
\begin{thm}[\cite{sta:aog}]
\label{pair}
For any graph $G=(V,E)$ with $\#V=n$ and any $t\in\bbP$ we have

\vs{8pt}

\eqed{
P(G;-t) =(-1)^n \cdot \#\{(O,\ka) \mid \text{$O$ and $\ka$ are compatible}\}.
}
\end{thm}

Note that this result implies Theorem~\ref{sta}.  For if $t=1$ then there is only one coloring $\ka:V\ra[1]$.  And this coloring is compatible with any acyclic orientation.  So in this case, the number of compatible pairs is just the number of acyclic orientations.  Sagan and Vatter~\cite{sv:bpp} have given a bijective proof of Theorem~\ref{pair}.

As far as Theorem~\ref{zas}, there is actually a stronger result which holds for any hyperplane arrangement.  Given an arrangement $\cH$, consider all the subspaces $S$ of $\bbR^n$ which can be formed by intersecting hyperplanes in $\cH$.  This includes $\bbR^n$ itself which is the empty intersection.  Partially order the subspaces by reverse inclusion to form a poset (partially ordered set) called the {\em intersection lattice} of $\cH$ and denoted $L(\cH)$.  Note that $\bbR^n$ is the minimum element of $L(\cH)$.  Given any finite poset $P$ with  a minimum there is an associated function
$$
\mu:P\ra\bbZ
$$
call the {\em M\"obius function} of $P$.  This map is a vast generalization of the M\"obius function from number theory and more information about it can be found in the texts of Sagan~\cite{sag:aoc} or Stanley~\cite{sta:ec1}.  We can now form the {\em characteristic polynomial} of $\cH$ which is the generating function
$$
\chi(\cH;t) = \sum_{S\in L(\cH)} \mu(S)\ t^{\dim S}.
$$
It turns out that if $\cH=\cH(G)$ for some graph $G$ then $\chi(\cH;t)=P(G;t)$.
Here is the full strength of Theorem~\ref{zas}.
\begin{thm}[\cite{zas:fua}]
For any hyperplane arrangement $\cH$ in $\bbR^n$ we have

\vs{8pt}

\eqed{
\chi(\cH;-1) = (-1)^n \cdot \#R(\cH).
}
\end{thm}

Hallam, Martin, and Sagan were able to improve on Theorem~\ref{peo}.  Let $G=(V,E)$ be a graph with $V=[n]$ and define
$$
\cI\cS\cF_k(G)=\{ F \mid \text{$F$ is an increasing spanning forest of $G$ with $k$ edges}\}
$$
so $\#\cI\cS\cF_k(G)=\isf_k(G)$.  Since $V=[n]$ and $E$ is a set of pairs of vertices, order $E$ lexicographically where each $e=ij\in E$ is listed with $i<j$.  Let
$$
\cN\cB\cC_k(G) = \{A\sbe E \mid \text{$A$ is an NBC set with $k$ edges}\}
$$
so that $\#\cN\cB\cC_k(G)=\nbc_k(G)$.
\begin{thm}[\cite{hms:isf}]
Let $G=(V,E)$ be a graph with $V=[n]$ and $E$ ordered lexicographically.  For all $k\in\bbN$ we have
$$
\cI\cS\cF_k(G)\sbe \cN\cB\cC_k(G)
$$
with equality for all $k$ if and only if the natural order on $[n]$ is a perfect elimination order.\hqed
\end{thm}

In~\cite{hms:isf} the authors also generalize both Theorems~\ref{Ej} and~\ref{peo} to certain pure simplicial complexes of dimension $d$.  When $d=1$, the graphical results are recovered.

\subsection{Chromatic symmetric functions}

Stanley~\cite{sta:sfg} generalized the chromatic polynomial to a symmetric function.  Let $\bx=\{x_1,x_2,x_3,\ldots\}$ be a countably infinite set of variables.  A formal power series $f(\bx)$ is {\em symmetric} if it is of bounded degree and invariant under permutations of the variables.  For example
$$
f(\bx)=7 x_1 x_2^2 + 7 x_2 x_1^2 + 7 x_1 x_3^2 + \cdots
-2 x_1 x_2 x_3 - 2 x_1 x_2 x_4 - 2 x_1 x_3 x_4 - \cdots
$$
is symmetric since all monomials of the form $x_i x_j^2$ have coefficient $7$, and all monomials of the form $x_i x_j x_k$ have coefficient $-2$.

\bfi
\begin{tikzpicture}
\draw(1,-.8) node{$P$};
\draw (0,0)--(1,2)--(2,0);
\fill(0,0) circle(.1);
\draw(-.5,0) node{$u$};
\fill(2,0) circle(.1);
\draw(2.5,0) node{$w$};
\fill(1,2) circle(.1);
\draw(1,2.5) node{$v$};
\end{tikzpicture}
\hs{20pt}
\begin{tikzpicture}
\draw(1,-.8) node{$\ka$};
\draw (0,0)--(1,2)--(2,0);
\fill(0,0) circle(.1);
\draw(-.5,0) node{$1$};
\fill(2,0) circle(.1);
\draw(2.5,0) node{$1$};
\fill(1,2) circle(.1);
\draw(1,2.5) node{$2$};
\end{tikzpicture}
\hs{20pt}
\begin{tikzpicture}
\draw(1,-.8) node{$\ka'$};
\draw (0,0)--(1,2)--(2,0);
\fill(0,0) circle(.1);
\draw(-.5,0) node{$2$};
\fill(2,0) circle(.1);
\draw(2.5,0) node{$2$};
\fill(1,2) circle(.1);
\draw(1,2.5) node{$1$};
\end{tikzpicture}
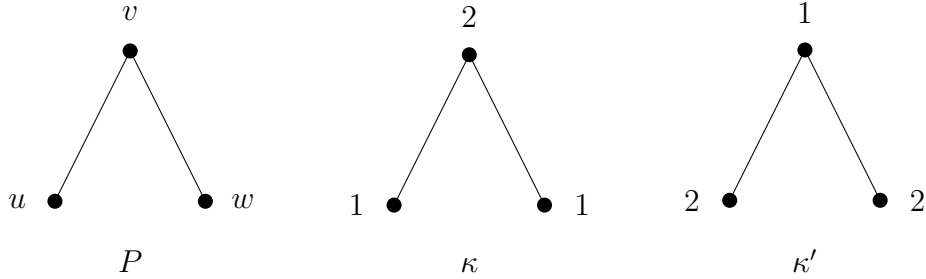
\capt{A path and two of its colorings \label{P}}
\efi

Consider colorings of a graph $G=(V,E)$ using the positive integers $\ka:V\ra\bbP$.  Associated with each such coloring is its {\em monomial}
$$
\bx^\ka =\prod_{v\in V} x_{\ka(v)}.
$$
Going back to our faithful example graph in Figure~\ref{G}, the middle coloring has $\bx^\ka=x_1^2 x_2^2$ while the one on the right has 
$\bx^{\ka'}=x_1 x_2^2 x_3$.  We now define the {\em chromatic symmetric function} of $G$ to be
$$
X(G)=X(G;\bx) = \sum_{\ka:V\ra\bbP} \bx^\ka
$$
where the sum is over proper colorings $\ka:V\ra\bbP$.  As an example, consider the path $P$ as shown in Figure~\ref{P}.  There are no proper colorings of $P$ with a single color.  Suppose we wish to use the two colors $1$ and $2$.  Then there are two possibilities as shown in the figure which contribute $\bx^\ka=x_1^2 x_2$ and $\bx^{\ka'}=x_1 x_2^2$ to $X(P)$.  The same argument shows that we get a term $x_i x_j^2$ for any distinct $i,j\in\bbP$.  Now consider using three colors on $P$, say $1$, $2$, and $3$.  Then any bijection $\ka:V\ra[3]$ is proper.  There are $6$ such maps for a contribution of 
$6x_1 x_2 x_3$.  Again, the choice of these three particular colors is immaterial so we get a term $6 x_i x_j x_k$ for any three positive integers $i,j,k$.  Putting things all together we obtain
$$
X(P) = \sum_{\text{$i,j$ distinct}} x_i^2 x_j 
+ 6 \sum_{\text{$i,j,k$ distinct}} x_i x_j x_k
$$
which is a symmetric function.  In general, $X(G)$ is symmetric because permuting colors in a proper coloring leaves it proper.  

Note also that $X(G;\bx)$ generalizes $P(G;t)$ in the following way.  Set 
\begin{equation}
\label{sub}
\text{$x_1=x_2=\ldots=x_t=1$ and  $x_i=0$ for $i>t$}.
\end{equation}
  Then each $\bx^\ka$ equals $0$ or $1$, and the latter happens only when $\ka$ uses colors in $[t]$.  So, under this substitution
$$
X(G;\bx) = \sum_{\ka:V\ra[t]} 1 = P(G;t)
$$
by definition of the chromatic polynomial.

One can now prove symmetric function generalizations of results about chromatic polynomials and also theorems about $X(G)$ which do not have analogues for $P(G)$.  To illustrate, we give an analogue of Whitney's NBC Theorem. Define
$$
c(G) = \text{number of components of $G$}.
$$
Note that if $G=(V,E)$ with $\#V=n$ and $A\in\cN\cB\cC_k(G)$ then $A$ is a forest, since if $A$ contained any cycle it would contain the corresponding broken circuit.  And $A$ is a spanning subgraph so $c(A)=n-k$.  Thus we can rewrite Theorem~\ref{whi} as
\begin{equation}
\label{NBCsum}
P(G;t) = \sum_{k=0}^n\ \sum_{A\in\cN\cB\cC_k(G)}\ (-1)^k\ t^{n-k}
=\sum_{\text{$A$ NBC}} (-1)^{\#A}\ t^{c(A)}.
\end{equation}

Symmetric functions form an algebra whose bases are indexed by {\em partitions}  which are weakly decreasing sequences
$\la=(\la_1,\la_2,\ldots,\la_l)$
of positive integers called {\em parts}.  Consider the {\em power sum} symmetric function basis defined multiplicitively by
$$
p_n = x_1^n + x_2^n + x_3^n +\cdots
$$
for $n\in\bbP$ and 
$$
p_\la = p_{\la_1} p_{\la_2} \cdots p_{\la_l}.
$$
To illustrate
$$
p_3 = x_1^3 + x_2^3 + x_3^3 + \cdots
$$
and
$$
p_{(3,3,1)} = p_3 p_3 p_1 =  (x_1^3 + x_2^3 + x_3^3 + \cdots)^2
(x_1 + x_2 + x_3 + \cdots).
$$
Note that using the substitution~\eqref{sub} we get $p_n=t$ and $p_\la=t^l$ where $l$ is the number of parts of $\la$.  Given any graph $G$ there is a corresponding partition $\la(G)$ whose parts are just the vertex sizes of the components of $G$.  As an example, in Figure~\ref{F} we have $\la(F)=\la(F')=(3,1)$.
The usual substitution shows that  the following result generalizes Whitney's Theorem in the form~\eqref{NBCsum},
\begin{thm}[\cite{sta:sfg}]
For any graph $G=(V,E)$ and any total order on $E$ we have

\vs{8pt}

\eqed{
X(G;x) = \sum_{\text{$A$ NBC}} (-1)^{\#A}\ p_{\la(A)}.
}
\end{thm}

However, the symmetric function does not always mirror the polynomial.  Let $T$ be a tree with $n$ vertices.  Then it is easy to see that for all such trees $P(T;t)=t(t-1)^{n-1}$.  It seems as if the opposite may be true for $X(T;t)$.  Call two graphs {\em isomorphic} if they yield the same graph once the labels on the vertices are removed. For example, the two forests in Figure~\ref{F} are isomorphic.  For more information about the following conjecture, the reader can consult~\cite{az:pcd,mmw:dtc,os:gec}.
\begin{conj}[\cite{sta:sfg}]
If $T$ and $T'$ are non-isomorphic trees then 
$$
X(T;\bx)\neq X(T';x).
$$
\end{conj}

Gessel~\cite{ges:mpp} introduced quasisymmetric functions which are an important refinement of symmetric functions.
Recently Shareshian and Wachs~\cite{sw:cqf} showed that there is a quasisymmetric refinement of $X(G;\bx)$.  This quasisymmetric function has important connections with Hessenberg varieties in algebraic geometry.

\subsection{Log-concavity}

A sequence of real numbers $a_0,a_1,\ldots,a_n$ is said to be {\em log-concave} if 
$$
a_k^2 \ge a_{k-1} a_{k+1}
$$
for all $0<k<n$.  As an example, the $n$th row of Pascal's Triangle
$$
\binom{n}{0}, \binom{n}{1},\ldots,\binom{n}{n}
$$
can be easily shown to be log-concave by using the formula for a binomial coefficient in terms of factorials.  Log-conave sequences abound in combinatorics, algebra, and geometry.  See the survey articles of Stanley~\cite{sta:lus}, Brenti~\cite{bre:lus}, and Br\"and\'en~\cite{bra:ulr}
for a plethora of examples.  We call a polynomial $p(t)=\sum_{k\ge0} a_k t^k$ {\em log-concave} if its coefficient sequence is.  In 2012, June Huh stunned the combinatorial world by using deep methods from algebraic geometry to prove the following result which generalizes a conjecture of Ronald Read~\cite{rea:icp} from 1968.
\begin{thm}[\cite{huh:mnp}]
For any graph $G$ we have that $P(G;t)$ is log-concave.\hqed
\end{thm}
\noindent  By developing a combinatorial version of Hodge theory, Adiprasito, Huh, and Katz~\cite{ahk:htc} were able to extend this result by proving the log-concavity of the characteristic polynomial of a matroid (a combinatorial object which generalizes both graphs and vector spaces).  



\nocite{*}
\bibliographystyle{alpha}

\end{document}